\documentclass[10pt,a4paper,oneside,fleqn,headinclude,footinclude]{scrartcl}                 
\usepackage[T1]{fontenc}                   
\usepackage[utf8]{inputenc}                 
\usepackage[english]{babel}        
\usepackage{graphicx}                       %
\usepackage[font=small]{quoting}            %
\usepackage{mparhack,fixltx2e,relsize}      %
\usepackage{comment}      
\usepackage{caption}                  %
\usepackage[nochapters,beramono,eulermath,%
            pdfspacing,
            listings,
            ]{classicthesis}                %
\usepackage{arsclassica}                    
\usepackage[margin=.9in]{geometry}
\usepackage[english]{babel}
\usepackage{verbatim}
\usepackage{color}
\usepackage{picture}
\usepackage{amsmath,amsthm,amsfonts,amssymb}
\usepackage{comment}
\usepackage{titling}

\newtheorem{theo}{Theorem}[section]

\theoremstyle{definition}

\newtheorem*{claim}{Claim}
\newtheorem{rem}{Remark}[section]

\numberwithin{equation}{section}

\newcommand{\R}{\mathbb R}
\newcommand{\Rn}{\mathbb R^n}

\newcommand{\de}{\partial}
\newcommand{\eps}{\varepsilon}

\newcommand{\wS}{\overline w_{\eps,\delta}}
\newcommand{\ws}{\underline w_{\eps,\delta}}
\newcommand{\lto}{\left(}
\newcommand{\rto}{\right)}

\DeclareMathOperator{\divergenza}{div}

\DeclareMathOperator{\sign}{sign}

\newtoks\by
\newtoks\paper
\newtoks\book
\newtoks\jour
\newtoks\yr
\newtoks\pg
\newtoks\vol
\newtoks\publ
\def\ota{{\hbox\vol{???}}}
\def\cLear{\by=\ota\paper=\ota\book=\ota\jour=\ota\yr=\ota
\pg=\ota\vol=\ota\publ=\ota}
\def\endpaper{\the\by. \textit{\the\paper},
{\the\jour} \textbf{\the\vol} (\the\yr), \the\pg.\cLear}
\def\endbook{\the\by, \textit{\the\book}, \the\publ.\cLear}
\def\endprep{\the\by, \textit{\the\paper}, \the\jour.\cLear}
\def\endproc{\the\by, \textit{\the\paper}, \the\book, \the\publ,
\the\yr, \the\pg.\cLear}

\begin{document}

\title
{Blow-up solutions for some nonlinear elliptic equations involving a Finsler-Laplacian}
\author
{Francesco Della Pietra%
\thanks{Universit\`a degli studi di Napoli Federico II,
Dipartimento di Matematica e Applicazioni ``R. Caccioppoli'',
Via Cintia, Monte S. Angelo - 80126 Napoli, Italia. Email: f.dellapietra@unina.it} 
{ and} 
Giuseppina di Blasio%
\thanks{Seconda Universit\`a degli studi di Napoli,
Dipartimento di Ma\-te\-ma\-ti\-ca e Fisica, 
Via Vivaldi, 43 - 81100 Caserta, Italia. Email: giuseppina.diblasio@unina2.it}
}
\date{\today}
\pagestyle{scrheadings} 
\maketitle
\begin{abstract}
In this paper we prove existence results and asymptotic behavior for strong solutions $u\in W^{2,2}_{\textrm{loc}}(\Omega)$ of the nonlinear elliptic problem
\begin{equation}
\tag{P}
\label{abstr}
\left\{
\begin{array}{ll}
-\Delta_{H}u+H(\nabla u)^{q}+\lambda u=f&\text{in }\Omega,\\
u\rightarrow +\infty &\text{on }\de\Omega,
\end{array}
\right.
\end{equation}
where $H$ is a suitable norm of $\R^{n}$, $\Omega$ is a bounded domain, $\Delta_{H}$ is the Finsler Laplacian, $1<q\le 2$, $\lambda>0$ and $f$ is a suitable function in $L^{\infty}_{\textrm{loc}}$. Furthermore, we are interested in the behavior of the solutions when $\lambda\rightarrow 0^{+}$, studying the so-called ergodic problem associated to \eqref{abstr}. A key role in order to study the ergodic problem will be played by local gradient estimates for \eqref{abstr}.\\

\noindent 
{\small {\itshape MSC 2010}: 35J60; 35J25; 35B44}\\
{\small {\itshape Keywords}: 
Anisotropic elliptic problems, Finsler Laplacian, Blow-up solutions}\\
\end{abstract}
\section{Introduction}
Let $\Omega$ be a $C^{2}$ bounded domain of $\R^{n}$, $n\ge 2$, and let us consider the following Finsler-Laplacian of $u$, namely the operator $\Delta_{H} u$ defined as
\[
\Delta_{H} u =\sum_{i=1}^{n}\frac{\de }{\de x_{i}}\big(H(\nabla u)H_{\xi_{i}}(\nabla u)\big),
\]
where $H$ is a suitable smooth norm of $\R^{n}$ (see Section \ref{2.1} for the precise assumptions). The aim of the paper is to study the existence of solutions of the equation
\begin{equation}
  \label{eq:1}
	-\Delta_H u + H(\nabla u)^q+\lambda u= f(x)\quad  \text{ in }\Omega
	\end{equation}
where $1<q\le 2$, $\lambda>0$ and $f$ is a suitable function in $L^{\infty}_{\textrm{loc}}$, bounded from below, with the boundary condition
\begin{equation}
\label{boundcond}
	\lim_{x\rightarrow\de \Omega} u(x)= + \infty.
\end{equation}
We will refer to the solutions of \eqref{eq:1} which satisfy \eqref{boundcond} as blow-up solutions. We are also interested in the asymptotic behavior of the solutions. Moreover, we study the behavior of the blow-up solutions of \eqref{eq:1} when $\lambda\rightarrow 0^{+}$.

Problems which deal with Finsler-Laplacian type operators have been studied in several contexts (see, for example, \cite{aflt,bfk,fk08,ciasal,wxjde,cfv,dpg4,pota,deth,j15}).

When $H$ is the Euclidean norm, namely $H(\xi)=|\xi|=\sqrt{\sum\xi_{i}^{2}}$, blow-up problems for equations depending on the gradient have been studied by many authors. We refer the reader, for example, to \cite{ll89,bg,gnr,pv,l07,p10,bpt,fgmp13}. In the Euclidean setting, problem \eqref{eq:1}-\eqref{boundcond} reduces to
\begin{equation}
  \label{eq:1ll}
 \left\{
 \begin{array}{l}
	-\Delta u + |\nabla u|^q+\lambda u= f(x)\quad  \text{ in }\Omega,\\
	\displaystyle\lim_{x\to \de\Omega}u(x)=+\infty.
 \end{array}
\right.
\end{equation}
The interest in problems modeled by \eqref{eq:1ll} has been grown since the seminal paper by Lasry and Lions \cite{ll89}. The equation in \eqref{eq:1ll} is a particular case of Hamilton-Jacobi-Bellman equations, which are related to stochastic differential problems. Indeed, in \cite{ll89} the authors enlightened the relation between problem \eqref{eq:1ll} and a model of stochastic control problem involving constraints on the state of the system by means of unbounded drifts. We briefly recall a few facts about this link. 

Let us consider the stochastic differential equation
\[
dX_{t}=a(X_{t}) dt+ dB_{t}\;, \qquad X_{0}=x\in \Omega,
\]
where $B_{t}$ is a standard Brownian motion. We assume that $a(\cdot)\in \mathcal A$, where $\mathcal A$ is the class of feedback controls such that the state process $X_{t}$, solution to the above SDE, remains in $\Omega$ with probability $1$, for all $t\ge 0$ and for any $x\in\Omega$. 
Thanks to the dynamic programming principle due to Bellman, the function $u_{\lambda}\in W^{2,r}_{\textrm{loc}}(\Omega)$, $r<\infty$ which solves \eqref{eq:1ll}  can be represented as the value function
\[
u_{\lambda} = \inf_{a\in\mathcal A} E\int_{0}^{\infty} \left[ f(X_{t})+ c_{q}|a(X_{t})|^{q'} \right]e^{-\lambda t}dt
\]
where $E$ is the expected value, $1<q\le 2$, $q'=\frac{q}{q-1}$, 
$c_{q}=(q-1)q^{-q'}$, and $e^{-\lambda t}$ is a discount factor. 

In \cite{ll89} there are several results regarding the existence, uniqueness and asymptotic behavior of the solutions of \eqref{eq:1ll}. 

When $\lambda$ tends to zero, the limit of $\lambda u_{\lambda}$ is known as ergodic limit. This kind of problems have been largely studied (see, for example, \cite{bf87,ll89,bf92,p10,fgmp13}). A typical result states that $\lambda u_{\lambda}$ tends to a value $u_{0}\in\R$ and $u_{\lambda}(x)-u_{\lambda}(x_{0})$, for fixed $x_{0}\in\Omega$, tends to a function $v$ which solves
\begin{equation}
  \label{eq:euclid-erg}
 \left\{
 \begin{array}{l}
	-\Delta v + |\nabla v|^q+u_{0}= f(x)\quad  \text{ in }\Omega,\\
	\displaystyle\lim_{x\to \de\Omega}v(x)=+\infty.
 \end{array}
\right.
\end{equation}
Problem \eqref{eq:euclid-erg} is seen as the ergodic limit, as $\lambda\to0^{+}$, of the stochastic control problem just destribed.

The scope of the present paper is to obtain existence, uniqueness and asymptotic behavior of the solutions to problem \eqref{eq:1}-\eqref{boundcond}, in the spirit of the work by Lasry and Lions \cite{ll89}, when $H$ is a general norm of $\Rn$. 

The interest in this kind of problems is twofold. First, in analogy with the relation between the quoted SDE and the elliptic problem \eqref{eq:1ll}, we stress that the Finsler Laplacian $\Delta_{H}$ can be interpreted as the generator of a ``h-Finslerian diffusion'', which generalizes the standard Brownian motion in $\R^{n}$. Stochastic processes of this type arise in some Biology problems, as in the theory of evolution by endo-symbiosis in which modern cells of plants and animals arise from separately living bacterial species. We refer the reader to \cite{antonelli1,antonelli2} (and to the bibliography cited therein) for the stochastic interpretation of $\Delta_{H}$ and for the quoted applications. Second, apart from the stochastic motivation, the nonlinear elliptic problem we study is of interest in its own right. In our case, the operator in  \eqref{eq:1} is, in general, anisotropic and quasilinear, with a strong nonlinearity in the gradient, and generalizes to this setting some extensively studied problems in the isotropic case. Actually, this brings several difficulties and differencies with respect to the Euclidean case. Moreover, in \cite{ll89} the asymptotic behavior of the solutions of \eqref{eq:1ll} near to the boundary of $\Omega$ is strongly related to a precise behavior of $f$ with respect to the distance to $\de\Omega$. In our case, the anisotropy of the operator leads to use an appropriate distance function to the boundary related to $H$. On the other hand, unless $H=|\cdot|$, the function $\nabla H(\xi)$ is always discontinuous in $\xi=0$. Hence, also giving smoothess assumptions on $H$ and on the data, it is not possible to apply classical Calder\'on-Zygmund type regularity results to get strong solutions in $W^{2,r}_{\textrm{loc}}(\Omega)$, $r<\infty$. We deal, in fact, only with solutions in $W^{2,2}_{\textrm{loc}}(\Omega)$. Furthermore, this lack of regularity does not permit to obtain, in general, the same gradient estimates for the solutions of \eqref{eq:1}-\eqref{boundcond} proved in the Euclidean case, which play a central role in the study of the ergodic problem. Actually, we are able to treat also the case $\lambda\rightarrow 0^{+}$, obtaining existence results for the limit problem
\begin{equation*}
 \left\{
 \begin{array}{l}
	-\Delta_{H} v + H(\nabla v)^q+u_{0}= f(x)\quad  \text{ in }\Omega,\\
	\displaystyle\lim_{x\to \de\Omega}v(x)=+\infty,
 \end{array}
\right.
\end{equation*}
and some properties of the ergodic constant $u_{0}$. We refer the reader to Section \ref{2.3} for the complete scheme of the obtained results. 

The paper is organized as follows. 

In Section \ref{2} we give the precise assumptions on $H$ and recall some basic facts of convex analysis. Moreover, we state our results. In Section \ref{3} we prove some a priori estimates for the gradient. Finally, in Section \ref{4} we give the proof of the main results.

\section{Assumptions, main results and comments}
\label{2}
\subsection{Notation and preliminaries}\label{2.1}
Throughout the paper we will consider a function 
\[
\xi\in \R^{n}\mapsto H(\xi)\in [0,+\infty[,
\] 
convex, $1$-homogeneous, that is
\begin{equation}
\label{1hom}
H(t\xi)=|t|H(\xi), \quad t\in \R,\,\xi \in \R^{n}, 
\end{equation}
 and such that
\begin{equation}
\label{crescita}
a|\xi| \le H(\xi),\quad \xi \in \R^{n},
\end{equation}
for some constant $0<a$. 
Under this hypothesis it is easy to see that there exists $b\ge a$ such that
\[
H(\xi)\le b |\xi|,\quad \xi \in \R^{n}.
\]
Moreover, we will assume that 
\begin{equation}
\label{strong}
H^{2}\in C^{3}(\R^n\setminus\{0\}),\text{ and }\nabla^{2}_{\xi}H^{2}\text{ is positive definite in }\R^{n}\setminus\{0\}.
\end{equation}
In all the paper we will denote with $\Omega$ a set of $\R^{n}$, $n\ge 2$ such that
\begin{equation}
\label{ip:omega}
\Omega\text{ is a bounded connected open set with }C^{2}\text{ boundary.}
\end{equation}

The hypothesis \eqref{strong} on $H$ assure that the operator $\Delta_{H}$ is elliptic, hence there exists a positive constant $\gamma$ such that
\begin{equation}
\label{convex}
\gamma |\xi|^{2}\le \sum_{i,j=1}^{n}{\dfrac{\de}{\de \xi_j}
  \big( H(\eta)H_{\xi_i}(\eta)\big)\xi_i\xi_j},
\end{equation}
for any $\eta \in \Rn\setminus\{0\}$ and for any $\xi=(\xi_{1},\ldots,\xi_{n})\in \R^n$.

We will consider as solutions of equation \eqref{eq:1} the strong solutions, namely functions $u\in W_{\textrm{loc}}^{2,2}(\Omega)$ such that the equality in $\eqref{eq:1}$ holds almost everywhere in $\Omega$. 

In this context, an important role is played by the polar function of $H$, namely the function $H^{o}$ defined as
\[
x\in \R^{n}\mapsto H^{o}(x)= \sup_{\xi \not=0}\frac{\xi\cdot x}{H(\xi)}.
\]
It is not difficult to verify that $H^{o}$ is a convex, $1$-homogeneous function that satisfies \eqref{crescita} (with different constants). Moreover,
\[
H(x)=\lto H^{o}\rto^{o}(x)= \sup_{\xi \not=0}\frac{\xi\cdot x}{H^{o}(\xi)}.
\]
The assumption \eqref{strong} on $H^{2}$ implies that $\{\xi\in\R^{n}\colon H(\xi)< 1\}$ is   strongly convex, in the sense that it is a $C^{2}$ set and all the principal curvatures are strictly positive functions on $\{\xi\colon H(\xi)=1\}$. This ensures that $H^{o}\in C^{2}(\R^{n}\setminus\{0\})$ (see \cite{schn} for the details).

The following well-known properties hold true:
\begin{align}
 &H_\xi(\xi)\cdot \xi = H(\xi),\quad \xi\ne 0, \label{eq:eul} \\
 &H_\xi ( t \xi ) = \sign t \cdot H_\xi(\xi), \quad \xi \ne 0, \, t \ne 0, \label{0-om}\\
 &\nabla^{2}_{\xi}H(t\xi)=\frac{1}{|t|}\nabla^{2}_{\xi}H(\xi)\quad\xi \ne 0, \, t \ne 0,\\
 & H\lto H_\xi^o(\xi)\rto=1,\quad \forall \xi \not=0, \label{eq:H1} \\
 &H^o(\xi)\, H_{\xi}\lto H^o_{\xi}(\xi) \rto = \xi \quad \forall \xi \not=0.  \label{eq:HH0}
\end{align} 
Analogous properties hold interchanging the roles of $H$ and $H^{o}$. 

The open set
\[
\mathcal W = \{  \xi \in \R^n \colon H^o(\xi)< 1 \}
\]
is the so-called Wulff shape centered at the origin.
More generally, we denote 
\[
\mathcal W_r(x_0)=r\mathcal W+x_0=\{x\in \R^{2}\colon H^o(x-x_0)<r\},
\] 
and $\mathcal W_r(0)=\mathcal W_r$.
\subsection{Anisotropic distance function}
Due to the nature of the problem, it seems to be natural to consider a suitable
notion of distance to the boundary. The anisotropic distance of $x\in\bar\Omega$ to the boundary of $\de \Omega$ is the function 
\begin{equation}
\label{defdist}
d_{H}(x)= \inf_{y\in \de \Omega} H^o(x-y), \quad x\in \bar\Omega.
\end{equation}
It is not difficult to prove that $d_H\in W^{1,\infty}(\Omega)$. Moreover, the property \eqref{eq:H1} gives that the $d_H(x)$
satisfies
\begin{equation}
  \label{Hd}
  H(\nabla d_H(x))=1 \quad\text{a.e. in }\Omega.
\end{equation}
Furthermore, if $\de\Omega$ is $C^2$, then $d_H$ is $C^2$ in a suitable neighborhood of $\de \Omega$ in $\bar\Omega$ 
(see \cite{cm07}).

Since $\de\Omega$ is $C^2$, it is possible to extend $d_H$ outside $\bar \Omega$ to a function which is still $C^2$ in a suitable neighborhood of $\de \Omega$ in $\R^n$. Indeed, let 
\[
\widetilde d_H(x)= \inf_{y\in \de \Omega} H^o(x-y), \quad x\in \R^n\setminus\Omega,
\]
and define the signed anisotropic distance function $d^s_H$ as
\begin{equation}
\label{disegno}
d^s_H(x)=\left\{
\begin{array}{ll}
d_H(x) &\text{ if }x\in \bar \Omega  \\[.1cm]
-\widetilde d_H(x) & \text{ if }x\in \R^n\setminus\bar\Omega.
\end{array}
\right.
\end{equation}
The following result is proved in \cite{cm07}.
\begin{theo}
\label{cmteo}
Let $\Omega$ be as in \eqref{ip:omega}. Then there exists $\mu>0$ such that $d^s_H$ is $C^2(A_\mu)$, with $A_\mu=\{x\in\R^n\colon -\mu<d^s_H(x)<\mu\}$.
\end{theo}
\subsection{Main results}\label{2.3}
The first result concerns the case when $f$ blows up at the boundary at most as $d_{H}(x)^{-q'}$, with $q'=q/(q-1)$.
\begin{theo}
\label{fpocodiv}
Let $f\in L^{\infty}_{\emph{loc}}(\Omega)$ bounded from below and such that
\begin{equation}
\label{condf1}
\lim_{{d_{H}(x)\rightarrow 0}} f(x) d_{H}(x)^{q'} = C_{1},\quad\text{ for some }0\le C_{1}<+\infty.
\end{equation}
Then, there exists a unique solution $u\in W^{2,2}_{\emph{loc}}(\Omega)$ of \eqref{eq:1} such that $u$ blows up at $\de \Omega$. Moreover, any subsolution $v\in W_{\emph{loc}}^{2,2}(\Omega)$ of \eqref{eq:1} is such that $u\ge v$ in $\Omega$. Finally, if $C_{0}$ is the unique positive solution of
$ \left(\frac{2-q}{q-1} \right)^{q}C_{0}^{q}-\frac{2-q}{(q-1)^{2}}C_{0}-C_{1}=0$ if $q<2$, $C_{0}^{2}-C_{0}-C_{1}=0$ if $q=2$, then
\begin{equation}
\label{roc}
u(x)\sim
\begin{cases}
  \dfrac{C_{0}}{d_{H}(x)^{\frac{2-q}{q-1}}} & \text{if }q<2,  \\[.6cm]
 {C_{0}\log\frac{1}{d_{H}(x)}} & \text{if }q=2,
\end{cases}
\end{equation}
as $d_{H}(x)\rightarrow 0$.
\end{theo}

The second main result we are able to prove is the case in which $f$
blows up very fast on $\de \Omega$.
\begin{theo}
\label{fmoltodiv}
	Let $\Omega$ be a bounded domain of $\R^{n}$, and suppose that $f\in L^{\infty}_{\emph{loc}}(\Omega)$ is bounded from below and satisfies
	\begin{equation}
	\label{fmoltodivcond}
	\liminf_{d_{H}\rightarrow 0} f(x)d_{H}^{\,\beta}(x) >0, \quad\text{for some }\beta\ge q'.
	\end{equation}
	Then, any solution $u\in W^{2,2}_{\emph{loc}}(\Omega)$ of \eqref{eq:1} bounded from below blows up at $\de \Omega$. Moreover there exists a maximum solution of \eqref{eq:1} in $W_{\emph{loc}}^{2,2}(\Omega)$ and, among all the solutions bounded from below in $\Omega$, there exists a minimum one which is the increasing limit of sequences of subsolutions of \eqref{eq:1}.
	
	If in addition there exists $C_{1}>0$ such that
	\begin{equation}
	\label{boh2}
f(x)\sim \frac{C_{1}}{d_{H}^{\,\beta}(x)},\quad \text{for some }\beta>q',
	\end{equation}
	then the blow up solution $u$ is unique and, as $d_{H}(x)\rightarrow 0$,
	\[
	u(x)\sim
  		\dfrac{C_{0}}{d_{H}(x)^{\frac \beta q -1}},
\]
with $C_{0}=(\alpha^{-1}C_{1})^{1/q}$.
\end{theo}
Finally, we prove what happens when $\lambda \rightarrow 0^{+}$. We will denote with $u_{\lambda}$ a blow up solution of \eqref{eq:1}, and  $v_{\lambda}=u_{\lambda}-u_{\lambda}(x_{0})$, where $x_{0}$ is any fixed point chosed in $\Omega$.
\begin{theo}
\label{teoerg}
Let $1<q\le 2$, and suppose that $f\in W^{1,\infty}_{\emph{loc}}(\Omega)$ is bounded from below and such that, as $d_{H}(x)\rightarrow 0$,
\begin{equation}
	\label{condfmd}
f(x) = o\left(\frac{1}{d_{H}(x)^{q'}}\right).
\end{equation}
Denote with $u_{\lambda}$ the unique solution of \eqref{eq:1} in $W^{2,2}_{\rm{loc}}(\Omega)$ such that $u_{\lambda}$ blows up at $\de \Omega$. Then, $\nabla u_{\lambda}$ and $\lambda u_{\lambda}$ are bounded in $L^{\infty}_{\rm{loc}}(\Omega)$ and $\lambda u_{\lambda}\rightarrow u_{0}\in \R$, $v_{\lambda}\rightarrow v\in W^{2,2}_{\emph{loc}}(\Omega)$ where the convergence is uniformly on compact sets of $\Omega$. Moreover, $v$ verifies \eqref{roc} and it is a solution of the ergodic equation
\begin{equation}
\label{eqerg}
-\Delta_{H} v+ H(\nabla v)^{q}+u_{0}=f\quad\text{in }\Omega.
\end{equation}
In addition, if $\tilde u_{0}$ is such that the equation $-\Delta_{H} w+ H(\nabla w)^{q}+\tilde u_{0}=f$ admits a blow-up solution in $W_{\rm loc}^{2,2}(\Omega)$, then necessarily $\tilde u_{0}=u_{0}$.
\end{theo}
We will refer to the unique constant $u_{0}$ such that \eqref{eqerg} admits a blow-up solution as the ergodic constant relative to \eqref{eqerg}.
\begin{rem}
We observe that the ergodic constant $u_{0}$, in the case $q=2$, is related to an eigenvalue problem. Indeed, if $v$ is a solution of the ergodic problem, performing the change of variable $w=e^{-v}$ and using the properties of $H$ we have that $w$ satisfies
\begin{equation}
\label{eigen-2}
\left\{
\begin{array}{ll}
	-\Delta_{H} w + f(x)\, w=u_{0}\,w &\text{in }\Omega,\\
	w=0&\text{on }\de\Omega,\\
	w>0&\text{in }\Omega.
\end{array}
\right.
\end{equation}
This observation will be useful in the proof of the uniqueness, up to an additive constant, of the blow-up solutions of \eqref{eqerg} (Theorem \ref{q=2theo} below). As a matter of fact, $u_{0}$ is the smallest eigenvalue of \eqref{eigen-2}. We refer to the proof of Theorem \ref{q=2theo} for the details. 

When $q\in]1,2[$, due to the nonlinearity of the principal part of the operator, and the fact that problem \eqref{eqerg} is non-variational, the uniqueness up to an additive constant of the solution of \eqref{eqerg} does not seem to be easy to prove.
\end{rem}
\begin{theo}
\label{q=2theo}
If $q=2$, under the hypotheses of Theorem \ref{teoerg}, and assuming also that $f\in W^{1,\infty}_{\emph{loc}}(\Omega)$ satisfies
\begin{equation}
\label{eq42}
\qquad |\nabla f(x)|\le \frac{C_{1}}{d_{H}^{3}(x)}
\end{equation}
for some $C_{1}\ge 0$,  if $v$ and $\tilde v$ are blow-up solutions in $W^{2,2}_{\emph{loc}}(\Omega)$ of \eqref{eqerg}, then $\tilde v =v+C$, for some constant $C\in\R$.  
\end{theo}
\section{Gradient bounds}
\label{3}
In this section we prove a local gradient bound for the solutions of
\begin{equation}
\label{eq:gb}
-\Delta_{H}u +H(\nabla u)^{q}+\lambda u=f,\quad u\in W_{\textrm{loc}}^{2,2}(\Omega).
\end{equation}
Such estimates are crucial in order to prove Theorem \ref{teoerg} on the ergodic problem. The method we will use relies in a local version, contained in \cite{ll89} (see also\cite{li80,li85}), of the classical Bernstein technique (see \cite{gt,ladyz}).

\begin{theo}
\label{gradbound}
Let $\Omega$ be a bounded open set, and suppose that $u\in W^{2,2}_{\emph{loc}}(\Omega)$ solves \eqref{eq:gb}. 
For any $\delta>0$, let us consider the set $\Omega_{\delta}=\{x\in\Omega\colon d_{H}(x)>\delta \}$. If $f\in C_{\emph{loc}}^{1,\vartheta}(\Omega)$, for some $\vartheta \in]0,1[$, then
\begin{equation}
\label{bellissima}
|\nabla u| \le C_{\delta}\quad\text{for any }x\in\Omega_{\delta},
\end{equation}
where the constant $C_{\delta}$ depends on $\|\nabla f\|_{\infty}$, $\sup(f-\lambda u)$, $\delta$ and $q$.
\end{theo}
Actually, we will prove in Section 4 that the estimate \eqref{bellissima} holds also under different assumptions on $f$ (see Remark \ref{remgradbound}).
\begin{proof}
The regularity assumptions on $H$ imply that $u\in C^{3}(\{\nabla u\ne 0\})\cap C^{1,\gamma}(\Omega)$ (see \cite{tk84,cfv,ciasal,ladyz}). 

For the sake of simplicity, we put 
 \[
 a^{ij}(\xi)=\frac 1 2 \left\{[H(\xi)]^{2} \right\}_{\xi_{i}\xi_{j}}.
 \]
  Hence the equation \eqref{eq:gb} can be written as (here and in the following the Einstein summation convention is understood)
\[
-a^{ij}(\nabla u)\,u_{x_{i} x_{j}} +[H(\nabla u)]^{q}+\lambda u = f.
\]
If $\nabla u \ne 0$, we can derive the equation with respect to $x_{k}$, obtaining that
\[
-a^{ij}u_{x_{i}x_{j}x_{k}}-a^{ij}_{\xi_{m}} \, u_{x_{m} x_{k}}\, u_{x_{i}x_{j}} + q\, H^{q-1} H_{\xi_{m}} u_{x_{m}x_{k}} +\lambda u_{x_{k}} =f_{x_{k}}.
\]
Let us consider $\varphi \in \mathcal{D}\left( \Omega \right) $ such that $0\leq \varphi \leq 1$ in $\Omega $, $\varphi \equiv 1$ on $\Omega _{\delta }$
and
\begin{equation}
\left\vert \Delta \varphi \right\vert \le C\,\varphi ^{\theta },\quad
\left\vert \nabla \varphi \right\vert ^{2}\le C\,\varphi ^{1+\theta }\quad \text{in }\Omega , \label{cond_phi}
\end{equation}
for some $\theta \in \left] 0,1\right[ $ that will be determined later and
some constant $C=C\left( \delta ,\theta \right)$. Multiplying by $\varphi\,u_{x_{k}}$ and summing we get
\begin{multline}
-a^{ij}u_{x_{i}x_{j}x_{k}}\varphi\,u_{x_{k}}
-a^{ij}_{\xi_{m}} \, u_{x_{m} x_{k}}\, u_{x_{i}x_{j}} \varphi\, u_{x_{k}} + q\, H^{q-1} H_{\xi_{m}} u_{x_{m}x_{k}} \varphi u_{x_{k}}+\lambda \varphi\, u_{x_{k}}u_{x_{k}} =\\=\varphi\, f_{x_{k}} u_{x_{k}}.\label{equality}
\end{multline}
Denoting $v=|\nabla u|^{2}$, equation \eqref{equality} can be rewritten as
\[
-a^{ij}\,v_{x_{i}x_{j}}\,\varphi
+2\varphi a^{ij}u_{x_{i}x_{k}}u_{x_{j}x_{k}}
-  a^{ij}_{\xi_{m}} \, u_{x_{i}x_{j}} \varphi\,v_{x_{m}}
+ q\, H^{q-1} H_{\xi}\cdot \nabla v \,\varphi+2\lambda \varphi\, v = 2\varphi\, \nabla f \cdot \nabla u,
\]
or
\begin{multline*}
-a^{ij}\,(\varphi\, v)_{x_{i}x_{j}}
+2\varphi a^{ij}u_{x_{i}x_{k}}u_{x_{j}x_{k}}
-  a^{ij}_{\xi_{m}} \, u_{x_{i}x_{j}} (\varphi\,v)_{x_{m}}
\\
  + q\, H^{q-1} H_{\xi}\cdot \nabla (v \,\varphi)+2\lambda \varphi\, v
 + \frac{2}{\varphi}\left(a^{ij}\varphi_{x_{i}} (\varphi v)_{x_{j}}\right)
 \\= 2\varphi\, \nabla f \cdot \nabla u + \left[ - a^{ij}_{\xi_{m}} \, u_{x_{i}x_{j}}\,\varphi_{x_{m}}+qH^{q-1}H_{\xi}\cdot \nabla \varphi \right]v -  a^{ij} \varphi_{x_{i} x_{j}} v + 2\frac{v}{\varphi} (a^{ij}\varphi_{x_{i}}\varphi_{x_{j}}).
\end{multline*}
Let $x_{0}$ be a maximum point for $\varphi v$ in $\Omega$. Obviously, $\nabla u(x_{0})\ne 0$, otherwise $\varphi v\equiv 0$ in $\Omega$. For the same reason, we can assume that $x_{0} \in {\textrm Supp}\,\varphi$. Then by the maximum principle we get the following inequality in $x_{0}$:
\begin{multline}
\label{eqmax}
2\varphi a^{ij}u_{x_{i}x_{k}}u_{x_{j}x_{k}}
 +2\lambda \varphi\, v
 \le 2\varphi\, \nabla f \cdot \nabla u +\left[- a^{ij}_{\xi_{m}} \, u_{x_{i}x_{j}}\,\varphi_{x_{m}}+qH^{q-1}H_{\xi}\cdot \nabla \varphi \right]v +\\-  a^{ij} \varphi_{x_{i} x_{j}} v + 2\frac{v}{\varphi} (a^{ij}\varphi_{x_{i}}\varphi_{x_{j}}).
\end{multline}

Now, being $H(\xi)$ $1$-homogeneous, and recalling that $a^{ij}=H\,H_{\xi_{i}\xi_{j}}+H_{\xi_{i}}H_{\xi_{j}}$, it follows that $a^{ij}_{\xi_{m}}$ are homogeneous of degree $-1$, and then 
\[
|\nabla u|\left|a^{ij}_{\xi_{m}}(\nabla u)\right|= \left|a^{ij}_{\xi_{m}}\left(\frac{\nabla u}{|\nabla u|}\right)\right|\le {C}.
\]
Hence, using the above inequality, the boundedness of $a^{ij}$, and the Young inequality, we get
\begin{multline*}
\left|a^{ij}_{\xi_{m}}(\nabla u)u_{x_{i}x_{j}}\varphi_{x_{m}}v\right| = |\nabla u| \left|a^{ij}_{\xi_{m}}\left(\frac{\nabla u}{|\nabla u|}\right)u_{x_{i}x_{j}}\varphi_{x_{m}}\right| \le
C|\nabla u||\nabla \varphi| \left(a^{ij}(\nabla u)\,u_{x_{i}x_{j}} \right)\le \\ \le
\eps \varphi \left(a^{ij}(\nabla u)\,u_{x_{i}x_{j}} \right)^{2}
 +C(\eps) \frac{|\nabla \varphi|^{2}}{\varphi} |\nabla u|^{2} = 
 \eps \,\varphi \left(H(\nabla u)^{q}+\lambda u -f\right)^{2}
 +C(\eps) \frac{|\nabla \varphi|^{2}}{\varphi} |\nabla u|^{2}.
\end{multline*}
On the other hand,
\[
a^{ij}(\nabla u)\,u_{x_{i}x_{k}}u_{x_{j}x_{k}} \ge \frac{(a^{ij}(\nabla u)\,u_{x_{i}x_{j}})^{2}}{{\textrm{trace}}[a^{ij}]} \ge C \left(H(\nabla u)^{q}+\lambda u -f\right)^{2}.
\]
Hence, for $\eps$ sufficiently small, recalling \eqref{eqmax} and that $\lambda u-f$ is bounded from below, we have
\begin{multline*}
\lbrack (H(\nabla u)^{q}-C_{1})^{+}]^{2}\varphi \le  \\
\le C\left\{ \frac{|\nabla \varphi |^{2}}{\varphi }|\nabla u|^{2}+2\varphi
\,\left\vert \nabla f\right\vert \left\vert \nabla u\right\vert +q\left\vert H\left( \nabla u\right) \right\vert ^{q-1}\left\vert H_{\xi }
\left( \nabla u\right) \right\vert \left\vert \nabla \varphi \right\vert v-a^{ij}\varphi_{x_{i}x_{j}}v+2\frac{v}{\varphi }(a^{ij}\varphi _{x_{i}}
\varphi_{x_{j}}).\right\}
\end{multline*}
Now using conditions \eqref{cond_phi}, \eqref{crescita}, the boundedness of $a^{ij}$ and the  $0$-homogeneity of $H_{\xi }$, we get
\begin{equation*}
\lbrack (H(\nabla u)^{q}-C_{1})^{+}]^{2}\varphi \le C\left(\varphi v^{\frac{1}{2}}+\varphi ^{\theta }v^{\frac{q+1}{2}}+\varphi^{\theta }v\right)
\end{equation*}
that means
\begin{equation*}
\varphi v^{q}\le C\left(1+\varphi v^{\frac{1}{2}}+\varphi ^{\theta }v^{\frac{q+1}{2}}+\varphi ^{\theta }v\right).
\end{equation*}
Easy computation show that if $\theta \ge \frac{3-p}{2}$, then
\begin{equation*}
\max_{\Omega }\varphi v=\varphi v\left( x_{0}\right) \le C.
\end{equation*}
Being $\varphi\equiv 1$ in $\Omega_{\delta}$, we get that
\[
|\nabla u|= v^{1/2}\le C_{\delta}\text{ in }\Omega_{\delta},
\]
and the proof is complete.
\end{proof}
Actually, we can prove a more precise estimate of the gradient of the solutions when we precise the behavior of the datum $f$ near the boundary.
\begin{theo}
\label{gradbound2}
Let $\Omega$ be a bounded open set, and suppose that $u\in W^{2,2}_{\emph{loc}}(\Omega)$ solves \eqref{eq:gb}. Supposing that $f\in W^{1,\infty}_{\emph{loc}}(\Omega)$ satisfies
\begin{equation}
\label{eq41}
|f(x)|\le \frac{C_{1}}{d_{H}^{\beta}(x)},\qquad
|\nabla f(x)|\le \frac{C_{1}}{d_{H}^{\beta+1}(x)}
\end{equation}
for some $\beta\le q'$, $C_{1}\ge 0$, 
and 
\[
\lambda u\ge -C_{2}
\]
for some $C_{2}\ge 0$. Then
\[
|\nabla u|\le \frac{C_{3}}{d_{H}^{\frac{1}{q-1}}(x)}\quad\text{in }\Omega,
\]
where $C_{3}$ only depends on $C_{1},C_{2},\beta$ and the diameter of $\Omega$. 
\end{theo}
\begin{proof}
Let $x_{0}\in \Omega$, define $r=\frac{1}{2}d_{H}(x_{0})$ and consider $v(x)=r^{\alpha}u(x_{0}+rx)$, $\alpha=\frac{2-q}{q-1}$, for $x\in \mathcal W_{1}(0)= \mathcal W$. The function $v\in W^{2,2}_{\textrm loc}(\mathcal W)$ solves 
\[
-\Delta_{H}v+H(\nabla v)^{q}+\lambda r^{2}v=r^{q'}f(x_{0}+rx) \quad\text{in }\mathcal W.
\]
The hypothesis \eqref{eq41} on $f$ gives that
\[
|r^{q'}f(x_{0}+rx)|\le C_{1}2^{\beta}r^{q'-\beta}\le C_{1}2^{\beta}[{\textrm diam}_{H}(\Omega)]^{q'-\beta}=C_{4}
\]
where ${\textrm diam}_{H}(\Omega)=\displaystyle\sup_{x,y\in\Omega} H^{o}(x-y)$, and, similarly,
\[
|r^{q'}\nabla_{x} f(x_{0}+rx)|\le C_{1}2^{\beta}r^{q'-\beta}\le C_{4}.
\]
 Now, using the estimate \eqref{bellissima}, we have
\[
|\nabla v(0)|=|\nabla u(x_{0})|r^{\frac{1}{q-1}}\le C_{3},
\]
where $C_{3}$ depends on $C_{4}$.
\end{proof}

\section{Proof of the main results}
\label{4}
\begin{proof}[Proof of Theorem \ref{fpocodiv}]
We split the proof considering first the case of $f$ bounded,
then we consider the general case, with $f\in L^{\infty}_{\textrm{loc}}(\Omega)$ such that \eqref{condf1} holds.

{\bf Case 1: $f\in L^{\infty}(\Omega)$}.
We look for solutions which blow up approaching to the boundary. To
this aim, we consider functions of the type $u(x)=C_0 d_H(x)^{-\alpha}$, with $C_0>0$ and $\alpha>0$. Recall that the anisotropic distance function is $C^2(\Gamma)$, where $\Gamma=\{x\in \bar\Omega\colon d_{H}(x) \le \delta_{0} \}$, with $\delta_{0}>0$ sufficiently small, is a tubular neighborhood of $\de \Omega$. If we substitute such functions in \eqref{eq:1}, by \eqref{1hom} and property \eqref{Hd} we get that 
\[
H(\nabla d_{H}^{-\alpha})= \alpha C_{0} d_{H}^{-\alpha-1}.
\]
Moreover, if $\bar y_{x}$ is the unique minimum point of \eqref{defdist}, that is $d_{H}(x)= H^{o}(x-\bar y_{x})$, then 
\[
\nabla d_{H}(x)=H_{\xi}^{o}(x-\bar y_{x})
\]
(see \cite[Prop. 3.3]{cm07}), and then by \eqref{eq:HH0} we have
\[
H_{\xi}(\nabla d_{H}(x))= H_{\xi}(H_{\xi}^{o}(x-\bar y_{x}))=\frac{x-\bar y_{x}}{H^{o}(x-\bar y_{x})}.
\]
Moreover, using \eqref{0-om}, we finally have 
\begin{equation}
\label{passaggio}
H_{\xi}(\nabla d_{H}^{-\alpha})=-H_{\xi}(\nabla d_{H})=-\frac{x-\bar y_{x}}{H^{o}(x-\bar y_{x})}.
\end{equation}
Hence, computing the anisotropic Laplacian and using \eqref{passaggio} and \eqref{eq:eul} it follows that
\begin{multline*}
\Delta_{H} \left(C_{0}\,d_{H}^{-\alpha}\right)= - C_{0}\,\alpha\, 
\divergenza\left[d_{H}(x)^{-\alpha-1}H_\xi(\nabla d_H(x))\right]= \\
= C_0 \alpha(\alpha+1) d_H(x)^{-\alpha-2}   \frac{H^o_\xi(x-\bar y_x)\cdot(x-\bar y_{x})}{H^{o}(x-\bar y_{x})} +\\- C_{0}\,\alpha \, d_H(x)^{-\alpha-1} \sum_{i,j=1}^{n} H_{\xi_i\xi_j}(\nabla d_H(x)) \de_{x_i x_j}d_H(x)= C_0 \alpha(\alpha+1) d_H(x)^{-\alpha-2} -K(x) d_H(x)^{-\alpha-1},
\end{multline*}
where 
\[
K(x)= C_{0}\,\alpha \sum_{i,j=1}^{n} H_{\xi_i\xi_j}(\nabla d_H(x)) \de_{x_i x_j}d_H(x)
\]
is bounded in $\Gamma$, being $H_{\xi\xi}$ bounded on $\{\xi\colon H(\xi)=1\}$, and $d_H\in C^2(\Gamma)$.
Hence 
\begin{multline}
  \label{eq:2}
 -\Delta_H u + H(\nabla u)^q + \lambda u -f =\\= -C_0 \alpha(\alpha+1)
 d_H^{-\alpha-2}  + K(x) d_H(x)^{-\alpha-1}+ C_0^{q}\alpha^q d_H^{-(\alpha+1)q} +\lambda
 C_0d_H^{-\alpha} -f.
\end{multline}
If $f$ is in $L^\infty$, the most explosive term in \eqref{eq:2} is
\begin{equation*}
  \label{alpha}
  -C_0 \alpha(\alpha+1) d_H^{-\alpha-2}+C_0^{q}\alpha^q d_H^{-(\alpha+1)q}.
\end{equation*}
If $q<2$, this leads to the choice of
\begin{equation}
	\label{alphaC0}
	\alpha= \frac{2-q}{q-1},\quad C_0= \frac 1 \alpha
	(\alpha+1)^\frac{1}{q-1},
\end{equation}
otherwise, for $q=2$, $u(x)=-C_{0}\log d_{H}$, and it leads to the choice
of $C_{0}=1$.

We construct, by means of the signed distance function $d^s_{H}(x)$, defined in \eqref{disegno}, a suitable family of subsolutions and supersolutions of \eqref{eq:1}. To this aim, recalling that $d^s_H(x)\in C^2(A_\mu)$, where $A_\mu$ is given in Theorem \ref{cmteo}, it is possible to construct a function $d(x)$ in $C^{2}(\R^{n})$ such that
\begin{equation}
\label{proldist}
\left\{
	\begin{array}{ll}
 		d(x)=d_{H}(x) & \text{if }x\in \bar \Omega,\text{ and }d_{H}(x)\le
			\delta_{0}, \\[.1cm]
		d(x) \ge \delta_{0} & \text{if }x\in \Omega,\text{ and }d_{H}(x)> 	
			\delta_{0}, \\[.1cm]
		d(x)=-\widetilde d_{H}(x)&\text{if }x\not\in \bar \Omega,\text{ and }\widetilde d_{H}(x)	
			\le \delta_{0}, \\[.1cm]
		d(x) \le -\delta_{0} & \text{if }x\not\in \Omega,\text{ and }
			\widetilde d_{H}(x)> \delta_{0},
	\end{array}
\right.
\end{equation}
where $\delta_0$ is a positive constant smaller than $\mu$.
Hence, if $q<2$, for $\eps\ge 0$ and $\delta$ such that $0\le \delta \le
\delta_{0}$, we define
\begin{equation}
\label{subsuper}
	\begin{array}{ll}
	\ws(x) = (C_{0}-\eps)(d(x)+\delta)^{-\alpha}
	 	- C_{\eps}\quad x\in \Omega^{\delta}, \\[.15cm]
	\wS(x) = (C_{0}+\eps)(d(x)-\delta)^{-\alpha} 		
		+ C_{\eps} \quad x \in \Omega_{\delta},
	\end{array}
\end{equation}
where $C_{\eps}$ is a constant which will be chosed later, and
\[
	\begin{array}{l}
		\Omega^{\delta}:=\{x\in \R^{n}\colon d(x)\ge -\delta \} \supset \Omega ,	\\[.1cm] 
		\Omega_{\delta}: =\{ x \in \Omega \colon d (x) > \delta \} \subset \Omega.
	\end{array}
\]
\begin{figure}[tb]
\centering
\includegraphics[scale=.34]{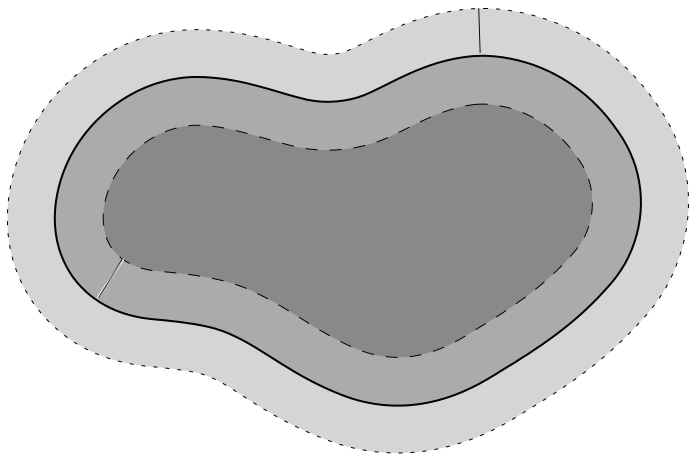}
  \put(-5cm,2.6cm){\footnotesize $\Omega_{\delta}$}
  \put(-2.1cm,1.44cm){\footnotesize $\Omega$}
  \put(-7cm,4.6cm){\footnotesize $\Omega^{\delta}$}
  \put(-2.5cm,5cm){\footnotesize $\delta$}
  \put(-7.3cm,2.1cm){\footnotesize $\delta$}
\end{figure}
If $q=2$, the functions $(d\pm \delta)^{-\alpha}$ in \eqref{subsuper} have to be
substituted with $-\log (d\pm \delta)$.

For suitable choices of $C_{\eps}$, the functions in \eqref{subsuper} are a
super and subsolution of \eqref{eq:1} in $\Omega^{\delta}$ and $
\Omega_{\delta}$, respectively (we may assume $f\equiv 0$ in $
\Omega^{\delta}\setminus\Omega$). Indeed, for $\alpha$ and $C_{0}$ as in \eqref{alphaC0}, 
we get
\begin{equation*}
\begin{split}
	-\Delta_{H} \wS&+H(\nabla \wS)^{q} + \lambda \wS - f = \\
	&= -\alpha(\alpha+1)(C_{0}+\eps)(d-\delta)^{-\alpha-2} H(\nabla d)^{2}
	  +\alpha( C_{0}+\eps )(d-\delta)^{-\alpha-1}\Delta_{H} d + \\
	&{ }\phantom{mm}+\alpha^{q}(C_{0}+\eps)^{q}(d-\delta)^{-q(\alpha+1)} H(\nabla d)^{q}+
	\lambda (C_{0}+\eps)(d-\delta)^{-\alpha} + \lambda C_{\eps} - f \ge \\
	&\ge \alpha(\alpha+1)(C_{0}+\eps)(d-\delta)^{-\alpha-2} 
	\left[\left(1+\frac{\eps}{C_{0}}  \right)^{q-1} H(\nabla d)^{q}-
	H(\nabla d)^{2} \right] +\\
	&{ }\phantom{spaziospiospaziospaziospaziospaziospazio}+
	\lambda C_{\eps}-\bar C(1+(d-\delta)^{-\alpha-1})\ge \\
	&\ge \nu \eps (d-\delta)^{-\alpha-2}+\lambda 
	C_{\eps}-C(1+(d-\delta)^{-\alpha-1}),
\end{split}
\end{equation*}
for some $\nu>0$ and $C>0$. We stress that $\Delta_H d=\frac 1 2 \divergenza ((H^2)_\xi(\nabla d))$ is bounded in $\Omega_\delta$ being $d\in C^2(\R^n)$ and $\nabla^2_\xi H^2\in L^\infty(\R^n)$. 

By choosing $C_{\eps}$ sufficiently large, last term in the above inequalities is nonnegative, and $\wS$ is a supersolution of \eqref{eq:1} in $\Omega_{\delta}$. The same argument shows that $\ws$ is a subsolution of \eqref{eq:1} in $\Omega^{\delta}$.
Now, fixed $M>0$, let us consider the approximating problem
\begin{equation}
	\label{sopra}
	\left\{
		\begin{array}{ll}
			-\Delta_{H}u_{M} +H(\nabla u_{M})^{q}+\lambda u_{M}= f&\text{in }\Omega,\\[.2cm]
			u_{M}=\underline w_{\eps, 1/ M}  &\text{on }\de\Omega.
		\end{array}
	\right.
\end{equation}
Observe that  $w_{\eps,\frac 1 M}=(C_{0}-\eps)M^{\alpha}-C_{\eps}=:C_{\eps,M}$ on $\de \Omega$. Then, 
performing the change of unknown $v_{M}=u_{M}- (C_{0}-\eps)M^{\alpha}-C_{\eps}$, problem \eqref{sopra} can be rewritten as
\begin{equation}
\label{soprazero}
	\left\{
		\begin{array}{ll}
			-\Delta_{H}v_{M} +H(\nabla v_{M})^{q}+\lambda v_{M}= f-\lambda C_{\eps,M} &\text{in }\Omega,\\[.2cm]
			v_{M}=0   &\text{on }\de\Omega.
		\end{array}
	\right.
\end{equation}
Problem \eqref{soprazero} admits a sub and supersolution in $L^{\infty}(\Omega)$ (it is sufficient to take two suitable constants).  Under hypotheses \eqref{crescita} and \eqref{convex}, by \cite[Theorem 2.1]{bmp1}, we get that problem \eqref{soprazero} admits a weak solution $v_{M}\in W_{0}^{1,2}(\Omega)\cap L^{\infty}(\Omega)$, namely $v_{M}$ satifies
\begin{multline*}
\int_{\Omega}\big[ H(\nabla v_{M})H_{\xi}(\nabla v_{M}) \cdot \nabla \varphi+ H(\nabla v_{M})^{q}\varphi+\lambda v_{M}\varphi\big]\,dx= \int_{\Omega} \left(f-\lambda C_{\eps,M}\right)\varphi\,dx,\\ \forall \varphi\in W_{0}^{1,2}(\Omega)\cap L^{\infty}(\Omega).
\end{multline*}
Then also \eqref{sopra} admits a weak solution $u_{M}\in W^{1,2}(\Omega)\cap L^{\infty}(\Omega)$. Moreover, such solutions are in $W^{2,2}(\Omega)$ (see \cite{tk84}, and the remarks contained in \cite{ciasal,cfv}), and in $C^{1,\vartheta}(\bar\Omega)$ (see \cite{ladyz,Lieberman1988}).
Now we apply the comparison principle contained in \cite[Theorem 3.1]{bbgk} (see also \cite[Theorem 3.1]{bm95}). We stress that the hypothesis (22) in \cite{bbgk} holds, because the function $(H(\xi) H_{\xi}(\xi))_{\xi}$ is $0$-homogeneous, and then
\[
(H(\xi) H_{\xi}(\xi))_{\xi}\, \xi - H(\xi) H_{\xi}(\xi)=0.
\]
Hence we have that, for $0<M<N$ and for any $\eps' > 0$,
\begin{equation}
\label{sopsot}
	\underline w_{\eps,\, 1/M} \le u_{M}\le u_{N} \le
	\overline{w}_{\eps',\,0}\quad \text{in }\Omega.
\end{equation}
 Last inequality in \eqref{sopsot} follows observing that $u_{N}<\overline w_{\eps',0}$ near the boundary of $\Omega$, being $u_{N}$ is finite on $\de \Omega$, while $\overline w_{\eps',0}|_{\de \Omega}=+\infty$, and then using the comparison principle. 
Hence \eqref{sopsot} gives that the functions $u_{M}$, $M>0$ are uniformly bounded in $L^{\infty}_{\textrm{loc}}(\Omega)$. This estimate, since $f\in L^{\infty}_{\textrm{loc}}(\Omega)$, allows to apply \cite[Theorem 1]{tk84} in order to obtain a $W^{1,\infty}_{\textrm{loc}}(\Omega)$ estimate. 
Actually, in any compact set $\Omega'\Subset \Omega$, by \cite[Theorem 1]{tk84} we have
\[
|\nabla u_{M}(x)-\nabla u_{M}(x')| \le C |x-x'|^{\vartheta},\quad \forall\, x,x'\in \Omega',
\]
where $C$ is a constant which depends only on $n, \gamma,\Gamma, \Omega', \vartheta$ and on the $L^{\infty}$ bound of $u_{M}$ in $\Omega'$. Then by Ascoli-Arzel\`a Theorem $u_{M}$, as $M\rightarrow +\infty$, converges locally uniformly to a function $\underline u\in C^{1}(\Omega)$. Moreover, $\underline u$ is a weak solution of \eqref{eq:1} and, recalling \eqref{sopsot},
 \begin{equation}
 \label{compa}
 \underline w_{\eps,0} \le \underline u \le \overline w_{\eps',0}, \quad
	\forall \eps'>0.
 \end{equation}

Using again \cite{tk84}, $\underline u\in W^{2,2}_{\textrm{loc}}(\Omega)\cap C^{1,\vartheta}_{\textrm{loc}}(\Omega)$. Then, by the chain rule for vector-valued functions contained in  \cite{mm72}, we have that $\underline u$ is a strong solution of \eqref{eq:1}. 

\textcolor{black}{ As matter of fact, $\underline u\ge \underline w_{\eps',0}$, for any $\eps'>0$.} By comparison principle, if $v\in W^{2,2}_{\textrm loc}(\Omega)$ is another solution of \eqref{eq:1} which blows up on the boundary, then $u_{M}\le v$. Hence, $\underline u$ is the minimal blow up solution. 

Next step consists in constructing a maximum blow-up solution of \eqref{eq:1}. To this end, we may argue as before to get the minimal solution $\underline u_{\delta}$ of \eqref{eq:1} in $\Omega_{\delta}$ which diverges on $\de\Omega_{\delta}$. We have that
\begin{equation}
\label{delirio}
	\ws \le \underline u_{\delta} \le \wS, \quad \forall \eps>0.
\end{equation}
Moreover, if $v\in W^{2,2}_{\textrm loc}(\Omega)$ is any blow up solution of \eqref{eq:1}, being $v$ bounded on $\de \Omega_\delta$, we have that 
\begin{equation}
\label{botte}
v\le \underline u_{\delta}.
\end{equation}
Passing to the limit as $\delta\rightarrow 0$ in \eqref{delirio}, using \eqref{compa} and \eqref{botte}, reasoning as before we get a maximal blow-up solution $\overline u=\lim_{\delta \to 0}\underline u_{\delta}$ of \eqref{eq:1} such that
\begin{equation}
	\label{catena}
	\underline w_{\eps,0} \le \underline u \le v\le  \overline u \le 	
	\overline w_{\eps,0}
\end{equation}
 for any $\eps>0$. As matter of fact, we claim that
 \[
	 \underline u=\overline u.
 \]
 Indeed, by \eqref{catena} it follows that
 \[
 	\lim_{d(x)\rightarrow 0} \frac{\overline u(x)}{\underline u(x)}= 1.
 \]
 Hence, being $\overline u(x)$ and $\underline u(x)$ divergent near the boundary, we get that for any $\theta \in ]0,1[$ there exists a neighbourhood of $\de \Omega$, dependent on $\theta$, in which
 \begin{equation*}
 \underline u(x) > \theta \overline u(x) + (1-\theta)\frac{m}{\lambda}=:w_{\theta}(x),
 \end{equation*}
 with $m=\inf_{\Omega} f$. The function $w_{\theta}$ is a subsolution of \eqref{eq:1}, and by maximum principle $w_{\theta}\le \underline u$ in all $\Omega$. As $\theta \rightarrow 1$, we get that $\overline u\le \underline u$ in $\Omega$, and we get the claim.
 
 We further emphasize that inequality \eqref{botte} clearly holds also if $v\in W^{2,2}_{\textrm{loc}}(\Omega)$ is any subsolution of problem \eqref{eq:1}. Passing to the limit, we get $v\le \underline u$. \\

{\bf Case 2:} $f$ unbounded. The proof runs analogously as in the previous case, except what concerns the existence of the minimum explosive solution. Indeed, if $f\sim C_{1} d_{H}(x)^{-q'}$ or $f=o(d_{H}^{-q'})$ near $\de \Omega$, substituting $u(x)=C_{0}d_{H}(x)^{-\alpha}$ in \eqref{eq:1}, with $\alpha= (2-q)/(q-1)$, we have that the most explosive term in \eqref{eq:2}, when $x$ approaches the boundary, is
\[
\left[ \left(\frac{2-q}{q-1} \right)^{q} C_{0}^{q}
	 - \frac{2-q}{(q-1)^{2}}C_{0}-C_{1} \right]	 d_H^{-q'}(x).
\]
Hence, as before we can construct a maximum explosive solution $\bar u$ of \eqref{eq:1} such that
\begin{equation}
	\label{bound2}
	(C_{0}-\eps) d^{-\alpha}- C_{\eps} \le \bar u \le (C_{0}+\eps)d^{-\alpha} 	+C_{\eps},
\end{equation}
 where $d$ is the function defined in \eqref{proldist}. As regards the existence of the minimum solution, differently from the bounded case we have that
$\underline w_{\eps,\delta}$ defined in \eqref{subsuper} is a subsolution of \eqref{eq:1} in $\Omega^{\delta}$, with $f$ replaced by
\[
f_\delta=
\begin{cases}
\min\left\{f, C_{2}+C_{3}(d+\delta)^{-q'}\right\} & \text{ in }\Omega,
\\[.2cm]
 C_{2}+C_{3}(d+\delta)^{-q'}& \text{ in }\Omega^{\delta}\setminus \Omega,
\end{cases}
\]
with $C_{2},C_{3}$ are positive constant such that $C_{3}>C_{1}$, and $C_{2}+C_{3}d^{{-q'}}>f$ in $\Omega$. Now, $f_{\delta}$ is bounded in $\Omega$, and from the first case we get that there exists a unique explosive solution $u_{\delta}$ of \eqref{eq:1} with $f$ replaced by $f_{\delta}$, and $u_{\delta} \ge \underline w_{\eps,\delta}$. Hence, being $f\ge f_{\delta}$, the comparison principle gives that $\bar u \ge u_{\delta}$. Passing to the limit, we obtain a minimal solution $\underline u(x)=\lim_{\delta\rightarrow 0}u_{\delta}(x)$ of \eqref{eq:1}, with $\underline u\le \bar u$, that satisfies \eqref{bound2}. Again, the uniqueness and the comparison with subsolutions follows as before.
\end{proof}
\begin{rem}
\label{remgradbound}
We observe that by taking a closer look to the proof of Theorem \ref{fpocodiv}, we are able  to conclude that the thesis of the Theorem \ref{gradbound} holds also if $f\in W^{1,\infty}_{\textrm{loc}}(\Omega)$ and \eqref{condf1} is satisfied. Indeed, by using the approximating problems
\begin{equation*}
	\left\{
		\begin{array}{ll}
			-\Delta_{H}\tilde u_{M} +H(\nabla \tilde u_{M})^{q}+\lambda \tilde u_{M}= f_{M}&\text{in }\Omega,\\[.2cm]
			\tilde u_{M}=\underline w_{\eps, 1/ M}  &\text{on }\de\Omega,
		\end{array}
	\right.
\end{equation*}
with $f_{M}$ sequence of smooth functions such that $f_{M}\rightarrow f$ in $W^{1,\infty}_{\textrm{loc}}(\Omega)$, the solutions $\tilde u_{M}$ are uniformly bounded in $L^{\infty}_{\textrm{loc}}(\Omega)$ and converge, up to a subsequence, to the unique blow-up solution $u$ of problem \eqref{eq:gb}. Then applying the bound \eqref{bellissima} in $\Omega_{\delta}$ to $\tilde u_{M}$ and passing to the limit we get the same bound also for $u$. 
\end{rem}
 
\begin{proof}[Proof of Theorem \ref{fmoltodiv}] 
The main part of the proof relies in the following statement.
\begin{claim}
If \eqref{fmoltodivcond} holds, then any solution of \eqref{eq:1} in $W^{2,2}_{\textrm{loc}}(\Omega)$, which is bounded from below, blows up when $d_{H}\rightarrow 0$.
\end{claim}

Once we prove the claim, the thesis of the theorem follows by adapting the proof contained in \cite[Theorems III.2 and III.3]{ll89} and the arguments used in Theorem \ref{fpocodiv} in order to construct a minimum and a maximum solution and, under the additional hypothesis \eqref{boh2}, that such solutions coincide.

In order to prove the claim, we may suppose, without loss of generality, that $u\ge 0$ in $\Omega$ and $f\ge \tilde K d_{H}^{-q'}$ for some positive constant $\tilde K$. Let $x_{0}$ be a point in $\Omega$ such that $d_{H}(x_{0})=2 r$. Hence $\mathcal W_{r} (x_{0})\Subset \Omega$, and from the equation we get that
	\[
	\left\{
	\begin{array}{ll}
		-\Delta_{H} u+H(\nabla u)^{q}+\lambda u \ge K r^{-q'}&
		\text{in }\mathcal W_{r}(x_{0})\\
	 	u\ge 0&\text{in }\de \mathcal W_{r}(x_{0}),
	\end{array}
	\right.
	\]
	where $K=2^{-q'}\tilde K$. This means that $u$ is a supersolution of
	\begin{equation}
	\label{above}
	\left\{
	\begin{array}{ll}
		-\Delta_{H} \tilde u_{r}+H(\nabla \tilde u_{r})^{q}+
		\lambda \tilde u_{r} = K r^{-q'}&
		\text{in }\mathcal W_{r}(x_{0})\\
	 	\tilde u_{r}= 0&\text{in }\de \mathcal W_{r}(x_{0})
	\end{array}
	\right.
	\end{equation}
 and, obviously, $w=0$ is a subsolution of \eqref{above}. Applying again \cite[Theorem 2.1]{bmp1} and \cite{tk84}, there exists a strong solution $\tilde u_{r}\in W^{2,2}(\mathcal W_{r}(x_{0}))\cap C^{1,\vartheta}(\overline{\mathcal W_{r}(x_{0})})$ of \eqref{above}.  Hence, $u(x) \ge \tilde u_{r}(x)\ge 0$ in $\mathcal W_{r}(x_{0})$. Defining $u_{r}(x)=r^{\alpha}\tilde u_{r}(rx+x_{0})$, for $x\in \mathcal W_{1}(0)=\mathcal W$, with $\alpha=(2-q)/(q-1)$, it follows that $u_{r}$ solves
		\begin{equation}
	\label{above1}
	\left\{
	\begin{array}{ll}
		-\Delta_{H} u_{r}+H(\nabla u_{r})^{q}+
		\lambda r^{2}u_{r} = K&
		\text{in }\mathcal W\\
	 	 u_{r}= 0&\text{in }\de \mathcal W.
	\end{array}
	\right.
	\end{equation}
For $k>0$, multiplying the above equation by $(u_{r}-k)^{+}$ and integrating, we easily get, by \eqref{eq:eul}, \eqref{crescita}, and being $u_{r}\ge 0$, that
\[
a \int_{u_{r}>k} |\nabla u_{r}|^{2} dx \le \int_{\mathcal W} H(\nabla u_{r}) H_{\xi}(\nabla u_{r})\cdot \nabla(u_{r}-k)^{+} dx \le K |\{u_{r}>k\}|,
\]
and, for $h>k$,
\[
|\{u_{r}>h\}| \le C (h-k)^{-2^{*}}|\{u_{r}>k\}|^{2^{*}/2},
\]
where $C$ is a constant independent on $r$. Hence, the classical Stampacchia Lemma (see \cite{stamp}) assures that $u_{r}$ is uniformly bounded in $L^{\infty}(\mathcal W)$.  Moreover, by \cite{bbgk} $u_{r}$ is the unique bounded solution of \eqref{above1}, which is also radial with respect to $H^{o}$, due to the symmetry of the data. That is, $u_{r}(x)=U_{r}(H^{o}(x))$, $x\in \mathcal W$.

Reasoning as in Theorem \ref{fpocodiv} we get that $u_{r} \rightarrow u_{0}\in W_{\textrm loc}^{2,2}(\mathcal W)$, where $u_{0}(x)=U_{0}(H^{o}(x))$, $x\in \mathcal W$ solves
\begin {equation*}
	\left\{
	\begin{array}{ll}
		-\Delta_{H} u_{0}+H(\nabla u_{0})^{q} = K &
		\text{in }\mathcal W,\\
	 	 u_{0}= 0&\text{in }\de \mathcal W.
	\end{array}
	\right.
	\end {equation*}
As a matter of fact, $U_{0}$ solves the problem
 \begin {equation*}
	\left\{
	\begin{array}{ll}
		-U_{0}''- \dfrac {n-1}{r} U_{0}' + |U_{0}'(r)|^{q} = K &
		\text{in }[0,1],\\[.3cm]
	 	 U_{0}(1)= 0,\; U_{0}'(0)=0. 
	\end{array}
	\right.
	\end {equation*}
Hence, by the maximum principle $U_{0}(0)=u_{0}(0)>0$. This implies that, for $q<2$, $u(x)$ diverges as $d_{H}\rightarrow 0$. As regards the case $q=2$,
	this method allows only to say that
	\[
	\liminf_{d_{H}\rightarrow 0} u \ge u_{0}(0)=K_{0}.
	\]
	As matter of fact, arguing as in \cite{ll89}, we have that for any $\eps>0$ there exists $s_{\eps}>0$ such that for $x\in \Omega$ with $d_{H}(x)<s_{\eps}$, then $u(x)\ge K_{0}-\eps$. Now, putting $v=u-(K_{0}-\eps)$, and repeating exactly the above argument for $v$ (at least for $2r<s_{\eps}$), we get that
	\[
	\liminf_{d_{H}\rightarrow 0} u\ge K_{0}+K_{0}-\eps=2K_{0}-\eps.
	\]
	Letting $\eps$ go to zero, and iterating the argument, we get that $u$ diverges as $d_{H}\to 0$ also if $q=2$.
\end{proof}

\begin{proof}[Proof of Theorem \ref{teoerg}]
The argument of the proof of Theorem \ref{fpocodiv} allows to obtain that the solution $u_\lambda$ of problem \eqref{eq:1} satisfies, if $1<q<2$,
\begin{equation}
	\label{bound}
	\frac{C_{0}-\eps}{d^{\alpha}} -\frac{C_{\eps}}{\lambda} \le u_{\lambda} \le
	\frac{C_{0}+\eps}{d^{\alpha}} +\frac{C_{\eps}}{\lambda},
\end{equation}
for all $\eps>0$, $\lambda\in ]0,1]$, and for some $C_{\eps}>0$. In the case $q=2$, the functions
$d^{-\alpha}$ have to be replaced with $|\log d|$. By \eqref{bound}, $\lambda u_{\lambda}$ is uniformly bounded from below and in $L^{\infty}_{\textrm{loc}}(\Omega)$. Moreover using Theorem \ref{gradbound} and Remark \ref{remgradbound} we get that also $\nabla u_{\lambda}$ is uniformly bounded  in $L^{\infty}_{\textrm{loc}}(\Omega)$. Then, $v_{\lambda}=u_{\lambda}(x)-u_{\lambda}(x_{0})$,  for some fixed $x_{0}\in\Omega$, is uniformly bounded with respect to $\lambda\in]0,1]$ in $W_{\textrm{loc}}^{1,\infty }(\Omega)$. Hence, for any $\Omega'\Subset\Omega$, there exists a constant $C_{\Omega'}$ independent on $\lambda$ such that
\[
|u_{\lambda}(x)-u_{\lambda}(x_{0})|\le C_{\Omega'}|x-x_{0}|.
\]
Passing to the limit we obtain, up to a subsequence, the convergence of $\lambda u_{\lambda }(x_{0})$ to a constant $u_{0}$ and of $\lambda v_{\lambda}$ to $0$. We finally prove that $v_{\lambda}$ converges to a blow-up solution of \eqref{eqerg}.
 First observe that $v_{\lambda }$ satisfies the following equation in $\Omega $:
\begin{equation}
\label{eq_v}
-\Delta _{H}v_{\lambda }+H\left(  \nabla v_{\lambda }\right) ^{q}+\lambda v_{\lambda }+\lambda u_{\lambda }\left( x_{0}\right) =f.
\end{equation}
 Hence, using again the arguments of the proof of the previous results, we can pass to the limit in \eqref{eq_v}, obtaining that $v_{\lambda}$ converges to a solution 
$v\in W^{2,2}_{\textrm{loc}}(\Omega)$ of the problem \eqref{eqerg}.

Now we prove a lower bound for $v$. Let $z=\frac{C_{1}}{d^{\alpha }}$, with $C_{1}\in ]0,C_{0}[$ fixed. Then in a sufficiently small inner tubular neighbourhood of $\de \Omega$, namely $\Omega\setminus \Omega_{\delta_{0}}$, we have that
\begin{equation*}
-\Delta _{H}z+H\left( \nabla z \right) ^{q}+\lambda
z\le f-\lambda u_{\lambda }\left( x_{0}\right).
\end{equation*}
 On the other hand, $v_{\lambda}$ is bounded from below in $\Omega_{\delta_{0}}$, namely there exists a constant $M\ge 0$ such that
\begin{equation*}
v_{\lambda }\ge -M\quad\text{ on }\quad\Omega _{\delta_{0}}.
\end{equation*}
Adapting the methods used in Theorem \ref{fpocodiv} it is possible to obtain that
\begin{equation}
\label{uff}
v_{\lambda }\ge -M+z= -M+\frac{C_{1}}{d^{\alpha }}\quad\text{  on  }\quad\Omega .
\end{equation}%
 Passing to the limit, also $v$ satisfies \eqref{uff}.
 
  Now we show that for any couple $(\tilde u_{0},\tilde v)$ of problem \eqref{eqerg}, with $\tilde v$ such that blows up at the boundary, $\tilde v$ diverges as in \eqref{roc}. 
To this aim, it is possible to consider $\overline w_{\eps,\delta}$ as in \eqref{subsuper} which is supersolution of the ergodic equation \eqref{eqerg} in 
$\Omega_{\delta}\setminus \Omega_{\delta_{0}}$, for some $0<\delta<\delta_{0}=\delta_{0}(\eps)$. Hence, by the comparison principle, and letting $\delta$ go to zero, we can conclude that
\begin{equation}
\label{boundtilde}
-C \le \tilde v \le \overline w_{\eps,0} + \max_{d=\delta_{0}(\eps)} |\tilde v| = (C_{0}+\eps)d^{-\alpha} + \max_{d=\delta_{0}(\eps)} |\tilde v| \quad\text{in }\Omega.
\end{equation}
Hence, $\tilde v$ is such that $-\Delta_{H}\tilde v+H(\nabla \tilde v)^{q} +\tilde v = g$, with $g=f-\tilde u_{0}+ \tilde v$. The bounds in  \eqref{boundtilde} and the condition \eqref{condfmd} assure that $g\in L^{\infty}_{\textrm{loc}}(\Omega)$ and also satisfies \eqref{condfmd}. By Theorem \ref{fpocodiv} we get that $\tilde v$ satisfies \eqref{roc}.

Now we show that if $(\tilde u_{0}, \tilde v)\in \mathbb R\times W^{2,2}_{\textrm{loc}}(\Omega)$ is a couple which solves \eqref{eqerg} and $\tilde v$ blows up at the boundary, then $\tilde u_{0}=u_{0}$ and $\tilde v= v+C$, for some constant $C\in\R$. 
 
As regards the uniqueness of the ergodic constant $u_{0}$, the proof runs similarly as in \cite{ll89}, supposing by contradiction that $u_{0}<\tilde u_{0}$. Let us choose $\eps>0$, and $0<\theta<1$. 
First, observe that obviously $v$ satisfies
\begin{equation}
\label{parpa}
-\Delta_{H} v +H(\nabla v)^{q} + \eps v = f+\eps v-u_{0}\quad\text{a.e. in }\Omega.
\end{equation}
On the other hand, we have from the $1$-homogeneity of $H$ that 
 \[
-\Delta_{H}(\theta \tilde v) +H(\nabla(\theta \tilde v))^{q} +\eps \theta \tilde v 
\le f +C(1-\theta) +\eps \theta \tilde v -\theta \tilde u_{0}.
\]
Moreover, since $v$ and $\tilde v$ diverge as $d^{-\alpha}$ near to the boundary of $\Omega$, then $\theta \tilde v\le C_{\theta}+ v$. Hence, from the above inequality it follows that
\begin{equation*}
\begin{split}
-\Delta_{H}(\theta \tilde v) +H(\nabla(\theta \tilde v))^{q} +\eps \theta \tilde v 
&\le 
f+\eps v-u_{0}+(u_{0}-\theta \tilde u_{0})+\eps C_{\theta} + C(1-\theta) \\
&\le  f+\eps v-u_{0}.
\end{split}
\end{equation*}
where last inequality holds for $\theta$ sufficiently near to $1$ and for $\eps=\eps(\theta)$ sufficiently small. Hence, $\theta \tilde v$ is a subsolution of
\eqref{parpa}. By Theorem \ref{fpocodiv}, $\theta \tilde v \le v$. As $\theta \to 1$, $\tilde v\le v$. This is in contradiction with the fact that 
any function of the type $\tilde v+c_{1}$, with $c_{1}\in \R$ solves the  ergodic problem with the same constant $\tilde u_{0}$.
\end{proof}

\begin{proof}[Proof of Theorem \ref{q=2theo}]
The hypothesis $q=2$ allows to perform a suitable change of variable.
Let $v\in W^{2,2}_{\textrm{loc}}(\Omega)$ be a solution of the ergodic equation \eqref{eqerg} with $v=\infty$ on $\de \Omega$. Then the function
\[
w=e^{-v}
\] 
belongs to $W^{1,\infty}_{0}(\Omega)\cap W^{2,2}_{\textrm{loc}}(\Omega)$. Let us verify that $|\nabla w|\in L^{\infty}$. Due the condition \eqref{condfmd}, we have that $C_{0}=1$ in \eqref{roc}, and then $|v|\le \log (d_{H}^{-1})$. Moreover, using also \eqref{eq42} we can apply Theorem \ref{gradbound2}, obtaining that $|\nabla v|d_{H}$ is bounded. Hence
\[
|\nabla w|= |\nabla v|e^{-v} \le C.
\]
Now, observe that using the properties of $H$ it holds that the function $w$ is a $W^{2,2}_{\textrm{loc}}(\Omega)\cap W^{1,\infty}(\Omega)$ solution of 
\begin{equation}
\label{autoval}
\left\{
\begin{array}{ll}
	-\Delta_{H} w + f(x)\, w=u_{0}\,w &\text{in }\Omega,\\
	w=0&\text{on }\de\Omega,\\
	w>0&\text{in }\Omega.
\end{array}
\right.
\end{equation}
The ergodic constant $u_{0}$ is a critical point of the Rayleigh quotient
\[
\mathcal R[\psi]=\dfrac{\displaystyle\int_{\Omega} H(\nabla \psi)^{2}dx+\int_{\Omega}f(x)\psi^{2}dx}{\displaystyle\int_{\Omega}\psi^{2}dx}.
\]
As a matter of fact, we claim that $u_{0}$ is the minimum eigenvalue, namely
\begin{equation*}
u_{0}=\min_{\substack{\psi\in W_{0}^{1,2}(\Omega)\\ u\ne 0}} \mathcal R[\psi],
\end{equation*}
and $u_{0}$ is the only eigenvalue associated to a positive eigenfunction. 
The claim follows observing, first of all, that being $f$ bounded from below, $f\ge -C$, the Rayleigh quotient $\mathcal R[\psi]$ satisfies
\[
\mathcal R[\psi] \ge -C.
\]
Then the existence of the minimum value of $\mathcal R[\psi]$ easily follows by using standand arguments of Calculus of Variations. Moreover the simplicity of $u_{0}$ and the fact that it is the unique eigenvalue associated to a positive eigenfunction follows by adapting the proof contained, for example, in \cite{pota} and \cite{klp}. Hence problem \eqref{autoval} admits, up to a multiplicative constant, a unique solution. This implies that if $v_{1}$ and $v_{2}$ solve \eqref{eqerg}, then $v_{1}$ and $v_{2}$ differ by a constant. 
\end{proof}

\paragraph*{Acknowledgements} 

The authors thank Alessio Porretta for some useful discussions on the topic.\vspace{.2cm}

\noindent This work has been partially supported by the FIRB 2013 project ``Geometrical and qualitative aspects of PDE's'' and by GNAMPA of INdAM.


\end{document}